\documentclass[leqno,11pt]{amsart}

\usepackage{latexsym, esint, color}
\usepackage{amsmath,amssymb, amsthm, amsfonts}
\theoremstyle{plain}

\newtheorem{theorem}{Theorem}[section]

\newtheorem{lemma}[theorem]{Lemma}
\newtheorem{proposition}[theorem]{Proposition}

\theoremstyle{definition}

\newcommand{\R}{\mathbb{R}}

\newcommand{\D}{\mathbb{D}}
\renewcommand{\d}{\partial}
\renewcommand{\div}{{\rm div}\,}

\newcommand{\ut}{\underline{\theta}}
\newcommand{\un}{\underline{\nu}}
\newcommand{\uk}{\underline{\kappa}}

\newcommand{\LpLinfty}{L_p(0,T;L_\infty)}
\newcommand{\LinftyLp}{L_\infty(0,T;L_p)}

\numberwithin{equation}{section}

\title{Stability of high-temperature viscous flows. A case of pudding model.}

\author[P. ~B.~Mucha]{Piotr B. Mucha}
\address{Piotr B. Mucha, University of Warsaw, Institute of Applied Mathematics and Mechanics,
Banacha~2, 02-097 Warsaw, Poland}
\email{pbmucha@mimuw.edu.pl}

\author[A.~\'{S}wierczewska-Gwiazda]{Agnieszka \'{S}wierczewska-Gwiazda}
\address{Agnieszka \'{S}wierczewska-Gwiazda, University of Warsaw, Institute of Applied Mathematics and Mechanics,
Banacha~2, 02-097 Warsaw, Poland}
\email{aswiercz@mimuw.edu.pl}

\begin{document}
\maketitle\maketitle

\begin{abstract}
We  investigate the Navier-Stokes-Fourier system for incompressible heat conducting inhomogeneous
fluid. The main result concerns existence of global in time regular large solutions, provided the initial
temperature is sufficiently large. The system may be viewed as a model of pudding, as we assume  the viscosity grows with the temperature.
\end{abstract}
\section{Introduction}
Our interest is directed to  incompressible inhomogeneous heat conducting fluids. The flow is described by the following equations
\begin{equation}\label{main}
 \begin{array}{lr}
  \varrho_t +v\cdot\nabla \varrho =0,& \mbox{ in \ \ } (0,T)\times\Omega,\\
  \varrho v_t + \varrho v \cdot \nabla v - \div (\nu (\theta) \D(v)) + \nabla \pi =0 & \mbox{ in \ \ } (0,T)\times\Omega,\\
 \div v=0 & \mbox{in \ \ } (0,T)\times\Omega,\\
\varrho\theta_t + \varrho v \cdot \nabla \theta - \div(\kappa(\theta) \nabla \theta) = \nu(\theta) \D^2(v) & \mbox{ in \ \ } (0,T)\times\Omega,
 \end{array}
\end{equation}
where $\varrho: (0,T)\times\Omega\to\R, v: (0,T)\times\Omega\to\R^3, \theta:(0,T)\times\Omega\to\R$ and $\pi:(0,T)\times\Omega\to\R$ represent the density, velocity, temperature and pressure respectively.  
Given functions $\nu:\R\to\R$ and $\kappa:\R\to\R$ are the viscosity and heat conductivity. By $\D(v)$ we mean the symmetric gradient of a velocity field. 
The system is supplemented with the boundary conditions
\begin{equation}\label{i2}
 v=0, \quad \frac{\d}{\d n} \theta =0 \qquad \mbox{ at }  (0,T)\times \d\Omega
\end{equation}
and initial data
\begin{equation}\label{i3}
 v|_{t=0}=v_0, \qquad \theta|_{t=0}=\theta_0, \qquad \varrho|_{t=0}=\varrho_0 \mbox{ \ \ \  at } \Omega.
\end{equation}

Our aim is to prove global in time existence of regular solutions without restrictions on the smallness of the data. The system (\ref{main}) is a modification of the classical
 Navier-Stokes-Fourier model \cite{BuMaSh2014}
for constant density fluids. Here we introduce a natural generalization on inhomogeneous 
fluids in order to justify the assumption of increasing viscous coefficient $\nu(\cdot)$. 
The occurrence of increasing viscosity is very natural for compressible fluids. And indeed,  the   variable density setting  brings the system closer to a compressible model, having for instance in mind 
the so-called
slightly compressible flows, cf.~\cite{Li96}. 
On the other hand, some motivations to study the above described system, although at first glance rather counterintuitive for incompressible fluids,  arise from such simple matters as e.g. pudding. 
 This definitely incompressible fluid, once heated, experiences the growth of  viscosity.

The system is thermodynamically isolated, for sufficiently
smooth solution the total energy  is conserved
\begin{equation}
\int_\Omega (\frac 12 \varrho| v|^2 + \varrho \theta)(t) \ dx =
\int_\Omega (\frac 12 \varrho_0 |v_0|^2 + \varrho_0 \theta_0) \ dx.
\end{equation}
Assuming that the initial temperature is sufficiently large, and as the temperature of the fluid 
 remains bounded  from below by the initial bound, then consequently the viscosity 
 will never be below the minimum of the initial viscosity.  This noteworthy observation 
 will play a key role in our  considerations.
\medskip 

System \eqref{main} is well investigated in the form for homogeneous fluids (with constant density).
In the literature  we find almost complete theory of existence of weak solutions for general structure 
\cite{BuFeMa2009} and some nontrivial results about regular solutions \cite{BuKaMa2011}. However
the theory is far from being complete. 
 Still the interplay between viscous term $\div \nu(\theta) \D(v)$
and energy balance term $\nu(\theta) \D^2(v)$ is not well understood. The present article is a step towards 
 this direction. We will observe that for high/large temperature $\theta$, the behavior of these terms
indeed leads to the stabilization of the system.  

\smallskip 

We restrict to initial densities, which are small perturbations of a constant density
\begin{equation}\label{ass1}
\|1-\varrho_0\|_{L_\infty((0,T)\times\Omega)}< c.
\end{equation}
Here $c$ denotes a small constant  compared to constants from estimates for linear systems -- 
see Section 2. What is important,  $c$ is independent of the solution, in particular of the initial data.
Clearly, the total mass of the fluid is preserved and
\begin{equation}\label{mass}
 \int_\Omega \rho(t) dx = \int_\Omega \rho_0 dx = {\rm mass}.
\end{equation}

The kernel of the system is a structure of the viscosity and heat conductivity,  we assume that 
\begin{equation}\label{ass2}
\nu (\theta)= \theta^m , \qquad m >0, \qquad \kappa(\theta) = (1+\theta)^l, \qquad l \geq 0.
\end{equation}
The growth of $\nu(\cdot)$ is necessary. We underline that the case  $l=0$ is also covered.

The main result of the present paper is the following.

\smallskip

\begin{theorem}\label{Th:main} Given $p>7$,
let $v_0\in W^{2-2/p}_p(\Omega)$, $\theta_0\in W^{2-4/p}_{p/2}(\Omega)$
and $\varrho_0\in L_\infty(\Omega)\cap W^1_p(\Omega)$, in addition $\varrho_0$  fulfills  
\eqref{ass1}. Assume that 
\begin{equation}\label{i4}
 \underline{\theta}= \inf \theta_0
\end{equation}
and  $\kappa(\cdot),\nu(\cdot)$ fulfill (\ref{ass2}). 

Then, provided $\ut$ is sufficiently large,  there exists a regular global in time solution to system (\ref{main}),  satisfying
\begin{equation}
\begin{split}
v\in L_\infty(0,\infty; W^{2-2/p}_p(\Omega))\cap W^{1,2}_p((0,\infty)\times\Omega),\\
\varrho\in L_\infty((0,\infty)\times \Omega)\cap C([0,\infty);W^1_p(\Omega)),\\
\theta\in L_\infty(0,\infty; W^{2-4/p}_{p/2}(\Omega))\cap W^{1,2}_{p/2}((0,\infty)\times\Omega).
\end{split}
\end{equation}
\end{theorem}

The techniques presented here arise from 
 maximal regularity estimates in the standard $L_p$-framework \cite{A1995,LSU1968,MuZa2000}. 
Our purpose, however, is not 
 the most sharp description of the regularity, 
  we rather 
 aim at capturing 
the most general constraints on $\nu(\cdot)$ and $\kappa(\cdot)$.  Our methods are close to the ones for the compressible Navier-Stokes equations \cite{Da2001, Mu2001, Mu2003, MuZa2000, MuZa2002, So90}.
 On the other hand, for the  
results on  large data solutions to the Navier-Stokes equations the reader might refer to \cite{Baetal2013, Mu2001a, Mu2008,  PoRaSiTi94, ZaZa2015} for constant density flows  
 and  to \cite{DaMu2013, PaZhZh2013} for inhomogeneous
models. 

The structure of the paper is the following. In the second section we state   linear problems,  along with various lemmas on regularity of solutions. We complete the section with the estimates for the temperature. All these facts are used in  Section~\ref{S:apriori}, where we provide crucial a priori  estimates that  allow to conclude that existence of regular solutions  is indeed global.  The last section contains the proof of existence of solutions and prescribes asymptotic behaviour of solutions (Theorem \ref{T:asympt}).


\bigskip

\section{Auxiliary  technical lemmas}

Throughout the paper we use the standard notation  \cite{A1995,LSU1968}. By $L_p(Q)$ we denote the Lebesgue space of functions integrable with $p$-th power. 
By $W^{1,2}_p((0,T)\times \Omega)$ we mean the Sobolev space equipped with the following norm
\begin{equation}
 \|u\|_{W^{1,2}_p((0,T)\times \Omega)} = \|u,u_t,\nabla_x^2u\|_{L_p((0,T)\times \Omega)}.
\end{equation}
The homogeneous seminorm related to this space is denoted as follows
\begin{equation}
 \|u\|_{\dot W^{1,2}_p((0,T)\times \Omega)} = \|u_t,\nabla_x^2u\|_{L_p((0,T)\times \Omega)}.
\end{equation}
Further, we simplify the notation $\nabla=\nabla_x$.

The fractional Sobolev-Slobodeckii space $W^{2-2/p}_p(\Omega)$ is defined as the trace space for the 
space $W^{1,2}_p((0,T)\times \Omega)$ taking the truncation of functions for time i.e. $\{t=0\}\times \Omega$.

\smallskip 

Firstly, in the series of lemmas  we collect estimates for linear problems.


\begin{lemma}\label{l:heat}[see e.g. \cite{LSU1968}]
 Let $\Omega$ be bounded,  $\alpha >0$ and  $f\in L_q((0,T)\times\Omega), u_0\in W^{2-2/q}_q(\Omega)$, then a solution to the  problem
 \begin{equation}\label{l1a}
  u_t-\alpha \Delta u =f \mbox{ in } (0,T)\times\Omega, \qquad u|_{t=0}=u_0
 \end{equation}
obeys the following estimate
\begin{equation}\label{l1b}
 \sup_t \|u\|_{W^{2-2/q}_q(\Omega)} + \|u_t,\alpha \nabla^2 u\|_{L_q((0,T)\times\Omega)} \leq C(\|f\|_{L_q((0,T)\times\Omega)}+\|u_0\|_{W^{2-2/q}_q(\Omega)}),
\end{equation}
where $C$ is independent of $T$.
The system \eqref{l1a} is considered with the no-slip condition
\begin{equation}\label{l1c}
 u=0 \mbox{ at }  (0,T)\times\d\Omega
\end{equation}
or 
with the homogeneous Neumann boundary relations
\begin{equation}\label{l1d}
 \frac{\d u}{\d n} =0 \mbox{ at }  (0,T)\times\d\Omega, \mbox{ \ \ in addition \ \ } \int_\Omega f(x,t) dx =0 \mbox{ for } t\in (0,T).
\end{equation}
In addition, if $u_0\equiv 0$, then
\begin{equation}\label{l1e}
 \sup_t \alpha^{1-1/q}\|u\|_{W^{2-2/q}_q(\Omega)} + \|u_t,\alpha \nabla^2 u\|_{L_q((0,T)\times\Omega )} \leq C\|f\|_{L_q((0,T)\times\Omega)}.
\end{equation}
\end{lemma}

\begin{lemma}\label{l:stokes}[see e.g.~\cite{La69, So64}]
 Let $\Omega$ be bounded, $\un>0, F\in L_p((0,T)\times \Omega )$, $v_0\in W^{2-2/p}_p(\Omega)$, then a solution to the Stokes system
 \begin{equation}\label{stokes}
 \begin{array}{lr}
  v_t - \un \Delta v + \nabla \pi = F & \mbox{in \ \ } (0,T)\times\Omega,\\
\div v=0 & \mbox{in \ \ } (0,T)\times\Omega,
 \end{array}
\end{equation}
with initial datum $v|_{t=0}=v_0$ and $v=0$ at the boundary satisfies the following bound
\begin{equation}\label{stokes1}
 \sup_t \|v\|_{W^{2-2/p}_p(\Omega)} + \|v_t,\un \nabla^2 v \|_{L_p((0,T)\times\Omega)} \leq C(\|F\|_{L_p((0,T)\times \Omega )}+\|v_0\|_{W^{2-2/p}_p(\Omega)}).
\end{equation}
In addition if $F\equiv 0$, then for a.a. $t\in (0,T)$
 \begin{equation}\label{l2a}
  \|v(t)\|_{W^{2-2/p}_p(\Omega)} \leq Ce^{- c \un t} \|v_0\|_{W^{2-2/p}_p(\Omega)}.
 \end{equation}
In the case $v_0\equiv 0$, we obtain
\begin{equation}\label{l2b}
\sup_t \un^{1-\frac{1}{p}} \|v\|_{W^{2-2/p}_p(\Omega)} + \|v_t,\un \nabla^2 v\|_{L_p((0,T)\times\Omega)}\le C\|F\|_{L_p((0,T)\times\Omega)}.
\end{equation}
\end{lemma}

\noindent
Estimates (\ref{l1e}) and (\ref{l2b}) are a simple consequence of rescaling the considered systems in time.


\smallskip 

The next observation concerns the behavior of the temperature. By the maximum principle we are able to 
control the lower bound of this quantity. This fact is crucial within our all considerations.

\begin{lemma}\label{est-theta}
 Let $\varrho, \theta,v$ be sufficiently smooth solutions to  system (\ref{main}), then ($\ut$ is given by (\ref{i4}))
\begin{equation}\label{i5}
 \theta(t,x) \geq \ut \mbox{ \ for  a.a. } (t,x) \in (0,T)\times\Omega.
\end{equation}
\end{lemma}

\begin{proof}
For given scalar function $u$ we define:  $(u)_-:=\min\{u,0\}$.   Multiplying \eqref{main}$_3$ with $(\theta-\ut)_-$,
 integrating over $\Omega$ we obtain
\begin{equation}\label{negative}
\int_{\Omega}\left(\frac{1}{2}\varrho\,[(\theta-\ut)_-^2]_t +\frac{1}{2}\varrho v\cdot\nabla [(\theta-\ut)_-^2]+\kappa(\theta)[\nabla(\theta-\ut)_-]^2-\nu(\theta)\D^2(v)(\theta-\ut)_-
 \right) dx=0.
\end{equation}
Using equation \eqref{main}$_1$ we have
$$
\frac{1}{2}\int_\Omega \varrho v\cdot\nabla [(\theta-\ut)_-^2] dx=-\frac{1}{2}\int_\Omega v\cdot \nabla \varrho 
 (\theta-\ut)_-^2 dx=\frac{1}{2}\int_\Omega \varrho_t  (\theta-\ut)_-^2 dx
$$
and since the last two terms in \eqref{negative} are nonnegative, thus
$$
\frac{d}{dt}\int_\Omega\varrho(\theta-\ut)_-^2 dx\le 0.
$$
Hence \eqref{i5} holds. 
\end{proof}

The next result is a modification of the standard energy law. It allows to control the average of  the quantity $\theta -\ut$ without dependence on $\ut$.

\begin{lemma}\label{l:mean}
 Let $\varrho, \theta, v$ be sufficiently smooth solutions to (\ref{main}), then
 \begin{equation}\label{i10}
  \int_\Omega (\frac 12 \varrho |v|^2 + \varrho (\theta -\ut))(t) dx =
  \int_\Omega (\frac 12 \varrho_0 |v_0|^2 + \varrho_0 (\theta_0 -\ut)) dx,
 \end{equation}
furthermore
\begin{equation}\label{i11}
 \int_\Omega (\theta -\ut) (t) dx \leq C
 \int_\Omega (\frac 12 \varrho_0 |v_0|^2 + \varrho_0 (\theta_0 -\ut)) dx.
\end{equation}
\end{lemma}

\begin{proof}
 Testing the momentum equation by $v$ and using the continuity equation we obtain
 \begin{equation}\label{i12}
  \frac{d}{dt} \int_\Omega\frac 12 \varrho |v|^2 dx + \int_\Omega \nu(\theta)\D^2(v) dx=0.
 \end{equation}
Next, integrating the heat equation with $\theta-\ut$, we get
\begin{equation}\label{i13}
 \frac{d}{dt} \int_\Omega \varrho(\theta -\ut) (t)dx =\int_\Omega \nu(\theta)\D^2(v) dx.
\end{equation}
Adding (\ref{i12}) and (\ref{i13}) we obtain (\ref{i10}).
Relation (\ref{i12}) provides  also that 
\begin{equation}\label{i14}
 \int_\Omega \frac{1}{2} \varrho(t) |v(t)|^2  dx \leq \int_\Omega  \frac{1}{2}\varrho_0 |v_0|^2 dx.
\end{equation}
Since the classical maximum principle for the continuity equation implies 
\begin{equation}\label{i15}
 \|\varrho(t) -1\|_{L_\infty} = \|\varrho_0-1\|_{L_\infty},
\end{equation} 
we conclude (\ref{i11}) directly from (\ref{i10}). We shall underline that information carried by (\ref{i11}) is important
as it  allows to control the whole norm of the temperature. We do not have direct
control on the average of the temperature. The explanation one can find in Theorem \ref{T:asympt}.
 
\end{proof}

\section{The a priori estimates}
\label{S:apriori}

Since we work in the framework of regular solutions, the kernel of our studies are a priori estimates. The construction/existence of solutions
is shown in Section 4.
 To use the methods of maximal regularity we   restate the system as follows
\begin{equation}\label{a1}
 \begin{array}{lr}
 \varrho_t+v\cdot\nabla \varrho=0,\\ 
  v_t - \un \Delta v + \nabla \pi =(1-\varrho)v_t-\varrho v\cdot \nabla v - \div ((\un - \nu(\theta))\D(v)) 
& \mbox{in \ \ } (0,T)\times\Omega,\\
\div v=0 & \mbox{in \ \ } (0,T)\times\Omega,\\
\theta_t  - \uk \Delta \theta =(1-\varrho)\theta_t+ \nu(\theta) \D^2(v) -\varrho v\cdot \nabla \theta - \div((\uk -\kappa(\theta)) \nabla \theta) & \mbox{in \ \ } (0,T)\times\Omega,
 \end{array}
\end{equation}
 where $\un=\nu(\ut)$ and $\uk=\nu(\ut)$. 

  In the first step we introduce the extension of the initial data. 
 To find suitable relations of  solutions in terms of $\ut$, $\uk$, $\un$ we  divide the sought functions into two parts 
\begin{equation}\label{a2}
 v= N + S,
\end{equation}
where $S$ is a solution to the Stokes system
\begin{equation}\label{a3}
 \begin{array}{lr}
  S_t - \un \Delta S + \nabla \pi = 0 & \mbox{in \ \ } (0,T)\times\Omega,\\
\div S=0 & \mbox{in \ \ } (0,T)\times\Omega,
 \end{array}
\end{equation}
with initial datum $S|_{t=0}=v_0$ and $S=0$ at the boundary.
By Lemma \ref{l:stokes} the solution satisfies
\begin{equation}\label{a4}
 \sup_t \|S\|_{W^{2-2/p}_p(\Omega)} + \|S_t,\un \nabla^2 S \|_{L_p((0,T)\times\Omega)} \leq C\|v_0\|_{W^{2-2/p}_p(\Omega)}.
\end{equation}
Thus in further considerations we treat  $S$ as a given vector field.
The remaining part  of the velocity field fulfills the system
\begin{equation}\label{a5}
\begin{array}{ll} 
N_t-\un\Delta N+\nabla \pi= & (1-\varrho)(N+S)_t\\
& -\varrho(N+S)\nabla (N+S)-\div((\un-\nu(\theta))\nabla(N+S)),\\[5pt]
\div N=0,\\[5pt]
N|_{t=0}=0.
\end{array}
\end{equation}
System (\ref{a5}) is considered with  zero Dirichlet conditions.
For the later use of  Lemma \ref{l:stokes},  case (\ref{l2b}), to estimate solutions, 
 we introduce the following quantity
\begin{equation}\label{a6}
 \Xi_{\un}^p(N):= \sup_t \un^{1-\frac{1}{p}} \|N\|_{W^{2-2/p}_p(\Omega)} + \|N_t,\un \nabla^2 N\|_{L_p((0,T)\times\Omega)},
\end{equation}
underlining the dependence on $\un$. The zero initial data for $N$ is crucial in proper definition of the above quantity.

\smallskip 

The same we perform for the temperature.
Let 
\begin{equation}\label{a7}
\theta=H+E+\ut,
\end{equation}
where $E$ solves the linear heat equation
\begin{equation}\begin{split} 
E_t-\uk\Delta E&=0,\\
 E|_{t=0}&=\theta_0-\ut
\end{split}\end{equation}
with the homogeneous Neumann boundary condition and by Lemma \ref{l:heat},  case (\ref{l1e})
\begin{equation}
\sup_t \|E\|_{W^{2-4/p}_{p/2}(\Omega)} + \|E_t,\uk \nabla^2 E \|_{L_{p/2}((0,T)\times\Omega)} \leq C\|\theta_0-\ut\|_{W^{2-4/p}_{p/2}(\Omega)}.
\end{equation}
Note that having  the boundary condition $\frac{\partial E}{\partial n}=0$ we control
the average of $E$ in terms of initial data
\begin{equation}
 \int_\Omega E(t) dx = \int_\Omega |E(t)| dx = \int_\Omega (\theta_0-\ut)dx.
\end{equation}

The remaining part is the solution   to the following problem
\begin{equation}\label{a9}
\begin{split}
H_t-\uk \Delta H&= (1-\varrho)(H+E)_t+\nu(\theta)\D^2(N+S)\\
&-\varrho(N+S)\nabla(H+E)-\div((\uk-\kappa(\theta))\D (H+E))=:F_1. \\
H|_{t=0}=0&
\end{split}
\end{equation}
again with the homogeneous Neumann boundary condition. 
Accordingly, keeping in mind Lemma \ref{l:heat}, case (\ref{l1e}),  we introduce
\begin{equation}\label{a10}
 \Xi_{\uk}^{p/2}(H):= 
\sup_t \uk^{1-\frac{2}{p}} \|H\|_{\dot W^{2-4/p}_{p/2}(\Omega)} 
+ \|(H-\{H\})_t,\uk \nabla^2 H\|_{L_{p/2}((0,T)\times\Omega)},
\end{equation}
where for given $w:(0,T)\times \Omega \to \R$ we use the notation
\begin{equation}
 \{w\}=\frac{1}{|\Omega|} \int_\Omega w(t) dx.
\end{equation}
We shall note that in the expression $ \Xi_{\uk}^{p/2}(H)$ we use the homogeneous norm $\dot W^{2-4/p}_{p/2}(\Omega)$.
This is a consequence of the fact that the space average of the right-hand side of  (\ref{a9})$_1$ is not necessarily equal to zero. For this reason we modify a solution up to a spatially homogeneous function.  Indeed, we consider a projection of 
(\ref{a9}) on the space with zero  average  in the $x$-space
\begin{equation}\label{a9-mean}
\begin{split}
(H-\{H\})_t-\uk \Delta (H-\{H\})&= F_1-\{F_1\},\\
H|_{t=0}&=0,
\end{split}
\end{equation}
to obtain that 
\begin{multline}\label{a10a}
 \Xi_{\uk}^{p/2}(H)= \Xi_{\uk}^{p/2}(H-\{H\})= 
 \\[5pt]
\sup_t \uk^{1-\frac{2}{p}} \|H-\{H\}\|_{\dot W^{2-4/p}_{p/2}(\Omega)} 
+ \|(H-\{H\})_t,\uk \nabla^2 H\|_{L_{p/2}((0,T)\times\Omega)}. 
\end{multline}
In order to recover the full norm we use Lemma \ref{l:mean} which controls the average of $\theta$.
We have
\begin{equation}\label{a10b}
 |\int_\Omega H(t)dx  | \leq \int_\Omega (\theta(t) -\ut - E(t)) dx \leq C\int_\Omega (\varrho_0 (\theta_0-\ut) +
 \frac 12 \rho_0 |v_0|^2)dx.
\end{equation}

Observe that the regularity of solutions to equations for velocity is set in the $L_p$-spaces whereas for the heat equation in the $L_{p/2}$-spaces. This naturally follows from the nonlinearity of the right-hand side  of $(\ref{main})_4$.
 Our  goal is not  the optimization of $p$, but we  concentrate our needs to have no restrictions on the magnitude of the initial configuration.

\smallskip 

Below we prove the  inequality,  which will play an essential role in obtaining  the a priori bounds.

\begin{proposition}\label{p:apr}
Let $\Xi^p_{\un}(N)$ and  $\Xi^{p/2}_{\uk}(H)$ be defined by \eqref{a6} and \eqref{a10} respectively. Then  there exists $\Phi_0\ge 0$ depending only on the initial data such that 
\begin{equation}\label{n0}
 \Xi_{\un}^p(N) + \Xi_{\uk}^{p/2}(H) \leq A_0K (\Xi_{\un}^p(N) + \Xi_{\uk}^{p/2}(H) )^2  + \Phi_0,
\end{equation}
for some constant $A_0\ge 0$ and  
\begin{equation}\label{n0a}
K= \max\left\{ \un^{-2+1/p}, 
\frac{\|\nu'(\theta)\|_{L_\infty((0,T)\times\Omega)}}{\un \,\uk^{1-1/p}},
\frac{\|\nu(\theta)\|_{L_\infty((0,T)\times\Omega)}}{\un^2},\un^{-1+1/p}\uk^{-1}, \frac{\|\kappa'(\theta)\|_{L_\infty((0,T)\times\Omega)}}{
\uk^{2-4/p} } \right\}
\end{equation}
is sufficiently small comparing to the size of initial data $(v_0,\theta_0-\ut)$, provided
\begin{equation}\label{n0b}
 \frac{\|\nu(\theta) - \un\|_{L_\infty((0,T)\times\Omega)}}{\un},  \frac{\|\kappa(\theta) - \uk\|_{L_\infty((0,T)\times\Omega)}}{\uk} \leq \frac{1}{40 C_0}, 
\end{equation}
where $C_0$ is the constant from Lemmas \ref{l:heat} and \ref{l:stokes}.

\end{proposition}

\begin{proof} The proof will consist of two steps, which correspond to the estimates for velocity and temperature solving systems \eqref{a5} and \eqref{a9} respectively. 

\smallskip 

{\it Step 1.} Estimating   the velocity field $N$, which solves system \eqref{a5} we will essentially use Lemma \ref{l:stokes}. Thus
\begin{equation}\label{n1}\begin{split}
 \Xi_{\un}^p(N)&\leq C_0\big[ \|(1-\varrho)N_t\|_{L_p((0,T)\times\Omega)} + \|(1-\varrho)S_t\|_{L_p((0,T)\times\Omega)} \\
& + \| \varrho N \cdot \nabla N,\quad \varrho N \cdot \nabla S, \quad \varrho S \cdot \nabla N, 
\quad \varrho S\cdot \nabla S\|_{L_p((0,T)\times\Omega)} \\
&+\frac{\|\nu(\theta)-\un\|_{L_\infty((0,T)\times\Omega)}}{\un}(\|\un \nabla^2 N\|_{L_p((0,T)\times\Omega)} +
 \|\un \nabla^2 S\|_{L_p((0,T)\times\Omega)})\\
&+\|\nu'(\theta) |\nabla H| |\nabla N|\|_{L_p((0,T)\times\Omega)} +
\|\nu'(\theta) |\nabla E| |\nabla N|\|_{L_p((0,T)\times\Omega)} \\
&+
\|\nu'(\theta) |\nabla H| |\nabla S|\|_{L_p((0,T)\times\Omega)}
 +\|\nu'(\theta) |\nabla E| |\nabla S|\|_{L_p((0,T)\times\Omega)}].
\end{split}\end{equation}
We collect estimates of each  term of the right-hand side of \eqref{n1}
\begin{equation}
 \|(1-\varrho)N_t\|_{L_p((0,T)\times\Omega)}\leq C\|1-\varrho\|_{L_\infty((0,T)\times\Omega)}
 \|N_t\|_{L_p((0,T)\times\Omega)}\le \frac{1}{40C_0}\Xi^p_{\un}(N),
\end{equation}
\begin{equation}
 \|(1-\varrho)S_t\|_{L_p((0,T)\times\Omega)}\leq C\|1-\varrho\|_{L_\infty((0,T)\times\Omega)}
 \|S_t\|_{L_p((0,T)\times\Omega)}\le \frac{1}{40C_0}\|v_0\|_{W^{2-2/p}_p(\Omega)}.
\end{equation}
Above we used the fact (\ref{ass1}) that 
$\|1-\varrho_0\|_{L_\infty(\Omega )}\le\frac{1}{40C_0}$.
At this point one understands the meaning of the smallness of  $c$ in   (\ref{ass1}). The constant
$c$ needs to be  small in comparison to constant 
$C_0$ appearing  in estimate (\ref{stokes1}) from Lemma \ref{l:stokes}. Next
\begin{equation}\label{n2}
\begin{split}
 \|\varrho N\cdot \nabla N\|_{L_p((0,T)\times\Omega)}&
 \leq C \|\varrho\|_{L_\infty((0,T)\times\Omega)}\|N\|_{L_\infty(0,T;L_p)} \|\nabla N\|_{L_p(0,T;L_\infty)}\\
 &\leq 
C \un^{-2+1/p} (\Xi^p_{\un}(N))^2\|\varrho_0\|_{L_\infty(\Omega)}.
\end{split}
\end{equation}
By entirely similar arguments we conclude also that
\begin{equation}\label{n3}
\begin{split}
 \|\varrho N\cdot \nabla S\|_{L_p((0,T)\times\Omega)} 
\leq 
C \un^{-2+1/p} \Xi_{\un}^p(N)\|v_0\|_{W^{2-2/p}_p(\Omega)}\|\varrho_0\|_{L_\infty(\Omega)},
\end{split}
\end{equation}
\begin{equation}\label{n4}
\begin{split}
 \|\varrho S \cdot \nabla N\|_{L_p((0,T)\times\Omega)}
 \leq 
C \un^{-1} \|v_0\|_{W^{2-2/p}_p(\Omega)} \Xi_{\un}^p(N)\|\varrho_0\|_{L_\infty(\Omega)}
\end{split}
\end{equation}
and 
\begin{equation}\label{n5}
\begin{split}
 \|\varrho S\cdot \nabla S\|_{L_p((0,T)\times\Omega)}
 \leq 
C \un^{-1} \|v_0\|^2_{W^{2-2/p}_p(\Omega)}\|\varrho_0\|_{L_\infty(\Omega)}.
\end{split}
\end{equation}
To justify the estimate
\begin{equation}\label{n6}
 \frac{\|\nu(\theta)-\un\|_{L_\infty((0,T)\times\Omega)}}{\un}\|\un \nabla^2 N\|_{L_p((0,T)\times\Omega)}  
\leq \frac{1}{40 C_0} \|\un \nabla^2 N\|_{L_p((0,T)\times\Omega)} 
\leq \frac{1}{40 C_0} \Xi_{\un}^p(N),
\end{equation}
we use (\ref{n0b}).
The constraint (\ref{n0b}) gives a clear explanation  why the growth of 
$\nu(\cdot)$ must be lower than exponential.
In the same way  (\ref{n0b}) is necessary to hold for the a priori estimate
\begin{equation}\label{n8}
 \frac{\|\nu(\theta)-\un\|_{L_\infty((0,T)\times\Omega)}}{\un}\|\un \nabla^2 S\|_{L_p((0,T)\times\Omega)} 
 \leq \frac{1}{40C_0} \|v_0\|_{W^{2-2/p}_p(\Omega)}.
\end{equation}
We stress that since $p>7$ we have 
 the embeddings $W^{1,2}_p ( (0,T)\times \Omega) \subset L_\infty((0,T)\times \Omega ) $ with suitable constant depending on $\un$ and
$\nabla W^{1,2}_{p/2}((0,T)\times \Omega) \subset L_\infty (0,T;L_p(\Omega))$, 
and indeed, the lower value of $p$ would not be possible. This allows to provide the estimates for the remaining terms
\begin{multline}\label{n9}
 \|\nu'(\theta) |\nabla H| |\nabla N|\|_{L_p((0,T)\times\Omega)} \leq C
\|\nu'(\theta)\|_{L_\infty((0,T)\times\Omega)} \|\nabla H\|_{\LinftyLp} \|\nabla N\|_{\LpLinfty}
\\
\leq C \frac{\|\nu'(\theta)\|_{L_\infty((0,T)\times\Omega)}}{\un} \uk^{-1+1/p}\Xi_{\uk}^{p/2}(H) 
\Xi_{\un}^p(N)
\end{multline} 
and again in the same manner we obtain
\begin{equation}\label{n10}
 \|\nu'(\theta) |\nabla E| |\nabla N|\|_{L_p((0,T)\times\Omega)} 
\leq C\frac{\|\nu'(\theta)\|_{L_\infty((0,T)\times\Omega)}}{\un} \|\theta_0 -\ut\|_{W^{2-4/p}_{p/2}(\Omega)}
 \Xi_{\un}^p(N),
\end{equation}
\begin{equation}\label{n11}
 \|\nu'(\theta) |\nabla H | |\nabla S|\|_{L_p((0,T)\times\Omega) } 
\leq C\frac{\|\nu'(\theta)\|_{L_\infty((0,T)\times\Omega)} }{\un} \uk^{-1+2/p} \Xi^{p/2}_{\uk}(H) 
\|v_0\|_{W^{2-2/p}_p(\Omega)},
\end{equation}
\begin{equation}\label{n12}
 \|\nu'(\theta) |\nabla E | |\nabla S|\|_{L_p((0,T)\times\Omega)} 
\leq C\frac{\|\nu'(\theta)\|_{L_\infty((0,T)\times\Omega)}}{\un} \|\theta_0-\ut\|_{W^{2-4/p}_{p/2}(\Omega)} \|v_0\|_{W^{2-2/p}_p(\Omega)}.
\end{equation}
Collecting the above estimates gives
\begin{multline}\label{n14}
 \Xi_{\un}^p(N)\leq C\big [ \un^{-2+1/p} (\Xi^p_{\un}(N) )^2
+ \frac{\|\nu'(\theta)\|_{L_\infty((0,T)\times\Omega)}}{\un} \uk^{-1+2/p}\Xi_{\uk}^{p/2}(H) \Xi_{\un}^p(N) + \\
+\frac{\|\nu'(\theta)\|_{L_\infty((0,T)\times\Omega)}}{\un} \|v_0\|_{W^{2-2/p}_p(\Omega)} \uk^{-1+2/p}\Xi_{\uk}^{p/2}(H) + 
\frac{\|\nu'(\theta)\|_{L_\infty((0,T)\times\Omega)}}{\un} \|\theta_0-\ut\|_{W^{2-4/p}_{p/2}(\Omega)} \Xi_{\un}^p(N)\\
+ \un^{-1} \|v_0\|^2_{W^{2-2/p}_p(\Omega)} + C\|v_0\|_{W^{2-2/p}_{p}(\Omega)} +
 \frac{ \|\nu'(\theta)\|_{L_\infty((0,T)\times\Omega)} }{\un} \|\theta_0-\ut\|_{W^{2-4/p}_{p/2}(\Omega)} \|v_0\|_{W^{2-2/p}_p(\Omega)} \big].
\end{multline}

Together with \eqref{n0b} once we  provide that the following quantities are sufficiently small,~i.e.
\begin{equation}\label{n13}
  \un^{-1+1/p} \|v_0\|_{W^{2-2/p}_p(\Omega)},\;
 \frac{\|\nu'(\theta)\|_{L_\infty((0,T)\times\Omega)}}{\un}, \;
\frac{\|\nu'(\theta)\|_{L_\infty((0,T)\times\Omega)}}{\un} \|\theta_0 -\ut\|_{W^{2-4/p}_{p/2}(\Omega)} \leq \frac{1}{40 C_0},
\end{equation}
we conclude
%

\begin{multline}\label{n14a}
 \Xi_{\un}^p(N)\leq C\big [ \un^{-2+1/p} \, (\Xi^p_{\un}(N) )^2
+ \frac{\|\nu'(\theta)\|_{L_\infty((0,T)\times\Omega)}}{\un\,  \uk^{1-2/p} } \, \Xi_{\uk}^{p/2}(H) \Xi_{\un}^p(N) + \\
+\frac{\|\nu'(\theta)\|_{L_\infty((0,T)\times\Omega)}}{\un \, \uk^{1-2/p}} \, \|v_0\|_{W^{2-2/p}_p(\Omega)} \Xi_{\uk}^{p/2}(H) + 
  C\|v_0\|_{W^{2-2/p}_p(\Omega)} \big].
\end{multline}

\bigskip

{\it Step 2}. In the second part we concentrate on the temperature. We have to recall first that because of the boundary conditions of  Neumann type we
do not control the very function $H$ for the estimates, but only $H-\{H\}$. Nevertheless,  Lemma \ref{l:mean} and (\ref{a10b}) allow to  control 
\begin{equation}
 |\int H(t) dx| \leq {\rm initial \; data}.
\end{equation}
Applying (\ref{l1e}) we find
\begin{equation}\label{n15}
\begin{split}
 \Xi_{\uk}^{p/2}(H) &\leq \big[ 
\| (1-\varrho)(H-\{H\})_t\|_{L_{p/2}((0,T)\times\Omega)}+\| (1-\varrho)E_t\|_{L_{p/2}((0,T)\times\Omega)}\\
&+\|\varrho N \cdot \nabla H\|_{L_{p/2}((0,T)\times\Omega)} + \|\varrho S \cdot \nabla H\|_{L_{p/2}((0,T)\times\Omega)}  \\
&+\|\varrho N\cdot \nabla E\|_{L_{p/2}((0,T)\times\Omega)} + \|\varrho S\cdot \nabla E\|_{L_{p/2}((0,T)\times\Omega)} \\
&+ \frac{\|\nu(\theta)\|_{L_\infty((0,T)\times\Omega)}}{\un^2} \|\un \nabla^2 N\|^2_{L_p((0,T)\times\Omega)} + \frac{\|\nu(\theta)\|_{L_\infty((0,T)\times\Omega)}}{\un^2}\|\un \nabla^2 S\|_{L_p((0,T)\times\Omega)}^2 \\
 &+\frac{\|\kappa(\theta) -\uk\|_{L_\infty((0,T)\times\Omega)}}{\uk} \|\uk \nabla^2 H\|_{L_{p/2}((0,T)\times\Omega)} + \frac{\|\kappa(\theta) -\uk\|_{L_\infty((0,T)\times\Omega)}}{\uk}\|\uk \nabla^2 E\|_{L_{p/2}((0,T)\times\Omega)}\\
&+\|\kappa'(\theta)\|_{L_\infty((0,T)\times\Omega)} \|\nabla H\|^2_{L_p((0,T)\times\Omega)} +\|\kappa'(\theta)\|_{L_\infty((0,T)\times\Omega)} \|\nabla H\|_{L_p((0,T)\times\Omega)}\|\nabla E\|_{L_p((0,T)\times\Omega)} \\
&+ 
\|\kappa'(\theta)\|_{L_\infty((0,T)\times\Omega)}\|\nabla E\|_{L_p((0,T)\times\Omega)}^2 \big].
\end{split}
\end{equation}
We estimate  the terms from the right-hand side of (\ref{n15}) step by step
\begin{equation}
 \|(1-\varrho)(H-\{H\})_t\|_{L_{p/2}((0,T)\times\Omega)}\leq C\|1-\varrho\|_{L_\infty((0,T)\times\Omega)}
 \|(H-\{H\})_t\|_{L_{p/2}((0,T)\times\Omega)}\le \frac{1}{40C_0}\Xi^{p/2}_{\uk}(H).
\end{equation}
Here we use the fact that 
$$
\left\{ (1-\varrho)\{H\}_t - \{(1-\varrho)\{H\}_t\} \right\} =0,
$$
since $\int_\Omega \varrho (t) dx$ is constant in time. The estimates of the next two terms follow in a simple way

\begin{equation}
 \|(1-\varrho)E_t\|_{L_{p/2}((0,T)\times\Omega)}\leq C\|1-\varrho\|_{L_\infty((0,T)\times\Omega)}
 \|E_t\|_{L_{p/2}((0,T)\times\Omega)}\le \frac{1}{40C_0}\|\theta_0-\ut\|_{W^{2-4/p}_{p/2}(\Omega)},
\end{equation}

\begin{equation}\label{n16}
\begin{split}
 \|\varrho N \cdot \nabla H\|_{L_{p/2}((0,T)\times\Omega)} 
&\leq C\|\varrho\|_{L_{\infty}( (0,T)\times\Omega)}\|N\|_{L_\infty(0,T; L_p)} \|\nabla H\|_{L_{p/2}(0,T;L_p)} \\
&\leq 
C \un^{-1+1/p} \uk^{-1}\|\varrho_0\|_{L_{\infty}(\Omega)} \Xi_{\un}^p (N) \Xi_{\uk}^{p/2}(H),
\end{split}
\end{equation}
whereas   the next term we  estimate  differently
\begin{equation}\label{n17}
\begin{split}
 \|\varrho S \cdot \nabla H\|_{L_{p/2}((0,T)\times\Omega)}& \leq C\|\varrho\|_{L_{\infty}( (0,T)\times\Omega)}\|S\|_{L_{p/2}(0,T;L_\infty)}
\|\nabla H\|_{L_\infty(0,T; L_{p/2})}  \\ 
&\leq C \|\varrho_0\|_{L_{\infty}(\Omega)}\| e^{-c\un t} \|_{L_{p/2}(0,T)} \|v_0\|_{W^{2-2/p}_p(\Omega)} \uk^{-1} \Xi_{\uk}^{p/2}(H)\\
& \leq 
C \un^{-2/p} \uk^{-1} \|\varrho_0\|_{L_{\infty}(\Omega)}\|v_0\|_{W^{2-2/p}_p(\Omega)} \Xi_{\uk}^{p/2}(H).
\end{split}
\end{equation}
In (\ref{n17}) we used (\ref{l2a}) from Lemma \ref{l:stokes} in order to obtain better information in the terms of $\un$. Next
\begin{equation}\label{n18}
 \|\varrho N\cdot \nabla E\|_{L_{p/2}((0,T)\times\Omega)} \leq C \un^{-1+1/p} \Xi_{\un}^{p} (N) \uk^{-1}\|\theta_0-\ut\|_{W^{2-4/p}_{p/2}(\Omega)}\|\varrho_0\|_{L_\infty(\Omega)},
\end{equation}
\begin{equation}\label{n20}
 \|\varrho S\cdot \nabla E\|_{L_{p/2}((0,T)\times\Omega)}\leq C \un^{-2/p} \|v_0\|_{W^{2-2/p}_p(\Omega)} \|\theta_0 - \ut\|_{W^{2-4/p}_p(\Omega)}\|\varrho_0\|_{L_\infty(\Omega)},
\end{equation}
\begin{equation}\label{n21}
\|\kappa'(\theta)\|_{L_\infty((0,T)\times\Omega)} \|\nabla H\|_{L_p((0,T)\times\Omega)}^2 \leq \|\kappa'(\theta)\|_{L_\infty} \uk^{-2+4/p} (\Xi^{p/2}_{\uk}(H))^2,
\end{equation}
\begin{equation}\label{n22}
\begin{split}
\|\kappa'(\theta)\|_{L_\infty((0,T)\times\Omega)} &\|\nabla H\|_{L_p((0,T)\times\Omega)}\|\nabla E\|_{L_p((0,T)\times\Omega)} \\
&
 \leq \|\kappa'(\theta)\|_{L_\infty((0,T)\times\Omega)} 
\uk^{-3/2+2/p} \Xi_{\uk}^{p/2}(H) \|\theta_0 - \ut\|_{W^{2-4/p}_{p/2}(\Omega)},
\end{split}\end{equation}
\begin{equation}\label{n23}
\|\kappa' (\theta)\|_{L_\infty((0,T)\times\Omega)} \|\nabla S\|_{L_p((0,T)\times\Omega)}^2 \leq \|\kappa' (\theta)\|_{L_\infty((0,T)\times\Omega)} \uk^{-1}  \|\theta_0 - \ut\|_{W^{2-4/p}_{p/2}(\Omega)}^2,
\end{equation}
\begin{equation}\label{n24}
\begin{split}
 \frac{\|\nu(\theta)\|_{L_\infty((0,T)\times\Omega)}}{\un^2}& \|\un \nabla^2 N\|^2_{L_p((0,T)\times\Omega)}
 \\&\le 
 \frac{1}{\un} \left(\frac{1}{\un}+ \frac{\|\nu(\theta) - \un\|_{L_\infty((0,T)\times\Omega)}}{\un}\right)\|\un \nabla^2 N\|^2_{L_p((0,T)\times\Omega)}
 \le \frac{C}{\un} (\Xi_{\un}^p(N))^2.
\end{split}\end{equation}

In analysis (\ref{n16})-(\ref{n24}) we used the following facts:
\begin{equation}
 \uk \nabla H \in L_{p/2}(0,T;L_\infty(\Omega)) \mbox{ \ and \ } \uk^{1-4/p} \nabla H \in L_\infty(0,T;L_p(\Omega))
\end{equation}
and 
\begin{equation}
 \uk \nabla E \in L_{p/2}(0,T;L_\infty(\Omega)) \mbox{ \ and \ } \nabla E \in L_\infty(0,T;L_p(\Omega)).
\end{equation}
On the other hand 
\begin{equation}
 \|X\|_{L_p(0,T;L_p)} \leq \|X\|_{L_{p/2}(0,T;L_p)}^{1/2} \|X\|_{L_\infty(0,T;L_p)}^{1/2}.
\end{equation}
We require  that $\ut$ is sufficiently large that (\ref{n0b}) 
holds. 
In addition we require that
\begin{equation}\label{n25}
 \un^{2/p}\,\kappa^{-1} \| v_0\|_{W^{2-2/p}_p(\Omega)}, \|\kappa'(\theta)\|_{L_\infty((0,T)\times\Omega)} \uk^{-3/2+2/p}\|\theta_0-\ut\|_{W^{2-4/p}_{p/2}(\Omega)} \leq \frac{1}{40 C_0}.
\end{equation}

Altogether gives us
\begin{multline}
 \Xi_{\uk}^{p/2}(H) \leq \frac{\|\nu(\theta) - \un\|_{L_\infty((0,T)\times\Omega)}}{\un^2} \, (\Xi_{\un}^p(N))^2+
\un^{-1+1/p}\uk^{-1}\, \Xi_{\uk}^{p/2}(H) \Xi_{\un}^p(N) \\
+ \frac{\|\kappa'(\theta)\|_{L_\infty((0,T)\times\Omega)}}{ \uk^{2-4/p} }\, (\Xi_{\uk}^{p/2}(H))^2 + 
\un^{-1+1/p} \uk^{-1} \|\theta_0-\ut\|_{ W^{2-4/p}_p(\Omega) } \Xi^p_{\un}(N)\\
+ \|\kappa'(\theta)\|_{L_\infty((0,T)\times\Omega)} \uk^{-1}\,  \|\theta_0 - \ut\|_{ W^{2-4/p}_{p/2}(\Omega) }^2 + \|\nu(\theta)\|_{ L_\infty((0,T)\times\Omega) }\un^{-2} \, \|v_0\|_{ W^{2-2/p}_p(\Omega)}^2 + 
C\|\theta_0-\ut\|_{ W^{2-4/p}_{p/2}(\Omega) }.
\end{multline}

Having these estimates we are able to prove the following 
\begin{multline}
 \Xi_{\un}^p(N) + \Xi_{\uk}^{p/2}(H) \leq \\ 
(\un^{-2+1/p} + \frac{\|\nu(\theta)\|_{L_\infty((0,T)\times\Omega)} }{\un^2})\,(\Xi_{\un}^p(N))^2 
 + (\frac{\|\nu'(\theta)\|_{L_\infty((0,T)\times\Omega)}}{\un \uk^{1-1/p}} +\un^{-1+1/p}\uk{-1})\, 
\Xi_{\uk}^{p/2}(H) \Xi_{\un}^p(N) \\
+ \frac{\|\kappa'(\theta)\|_{L_\infty((0,T)\times\Omega)}}{\uk^{2-4/p} }\, (\Xi_{\uk}^{p/2}(H))^2
+\Phi (\|v_0\|,\|\theta_0-\ut\|),
\end{multline}
what immediately implies (\ref{n0}) and thus Proposition \ref{p:apr}  is proved. The smallness of $K$ is defined by (\ref{n13}) and (\ref{n25}).
\end{proof}

Finally we need  to show that Proposition \ref{p:apr} gives the a priori bound for solutions as $\ut$ is sufficiently large.

\smallskip

\begin{proposition}
 Let $(\varrho_0, v_0,\theta_0-\ut)\in L_\infty(\Omega) \times W^{2-2/p}_p(\Omega) \times W^{2-4/p}_{p/2}(\Omega)$. If 
$\ut$ is sufficiently large then
\begin{equation}\label{k0}
 \Xi_{\un}^p(N) + \Xi_{\uk}^{p/2}(H) \leq 2 \Phi_0(\varrho_0, v_0,\theta_0-\ut).
\end{equation}
\end{proposition}

\begin{proof}
By Proposition \ref{p:apr} we have
\begin{equation}\label{}
 \Xi_{\un}^p(N) + \Xi_{\uk}^{p/2}(H) \leq A_0 K (\Xi_{\un}^p(N) + \Xi_{\uk}^{p/2}(H) )^2  + \Phi_0,
\end{equation}
for all $T>0$. Is is clear that for small $T$
\begin{equation}
 \Xi_{\un}^p(N) + \Xi_{\uk}^{p/2}(H) \leq 2\Phi_0,
\end{equation}
even  without any restriction on the smallness of  $K$. We show however that if $K$ is small then $T$ can be infinite.  Note that as 
\begin{equation}\label{K}
 K \leq \frac{1}{4A_0 (\Phi_0+1)},
\mbox{ \ then  \ \ }
 \Xi_{\un}^p(N) + \Xi_{\uk}^{p/2}(H) < 2\Phi_0.
\end{equation}
and by a  simply contradiction argument one shows that $T=\infty$  as $K$ fulfills (\ref{K}).

Let us describe the dependence between $\ut$ and $\Phi_0$ in order to guarantee \eqref{K}. Consider the term
\begin{equation}\label{k1}
 \frac{\nu (\theta)-\un}{\un} \sim \frac{(\Phi_0 + \ut)^m-\ut^m}{\ut^m} \sim \frac{\Phi_0}{\ut}.
\end{equation}
Similarly we check  the behavior of the  term $\frac{\nu(\theta)}{\un^2}$. And
\begin{equation}
 \frac{\kappa'(\theta)}{\uk^{2-4/p}}\sim \frac{ (\Phi_0+\ut)^{l-1} }{ \ut^{10/7\, l} } \sim \frac{1}{\ut^{3/7\,l + 1}}\frac{ (\Phi_0+\ut)^{l-1}}{\ut^{l-1}}.
\end{equation}
We observe that as 
\begin{equation}
 \frac{\Phi_0}{\ut} \ll 1,
\end{equation}
then (\ref{K}) is fulfilled. Note that the smallness of the above quantity is of the same type as (\ref{k1}). We claim (\ref{k0}) holds for all $T$, provided (\ref{K}). 

\end{proof}

\section{Existence and asymptotic behaviour of solutions}

\begin{proof}[Proof of Theorem~\ref{main}]
The scheme of existence of solutions to  system (\ref{main}) follows from a modification of the approach to the inhomogeneous Navier-Stokes equations \cite{DaMu2009}.
The obstacle is the hyperbolic character of the continuity equation resulting into problems in direct application of the Banach iteration procedure. The solution 
of this problem is the following. We construct a sequence of approximative solutions to (\ref{main}) by a natural iteration procedure. We show the sequence is
uniformly bounded in spaces defined by the a priori estimates from Section \ref{S:apriori} and we accomplish by showing  that the sequence is indeed a Cauchy sequence in a larger space 
of the $L_2$-type. 

To avoid complex language of Lagrangian coordinates like in \cite{DaMu2012} we stay in the Eulerian coordinates framework. Hence in order to perform an existence and uniqueness result
we need to add an  assumption of the regularity of the initial density, i.e. $\nabla \rho_0 \in L_p(\Omega)$.

We construct approximate solutions inductively. Let $(\varrho^0,  v^0, \theta^0):=
(\varrho_0, v_0,\theta_0)$ and $(\varrho^k,  v^k, \theta^k)$ satisfy in $(0,T)\times\Omega$
\begin{equation}
\begin{split}\label{app}
&\varrho^k_t+v^{k-1}\nabla\varrho^k=0,\\
&\varrho^k v^k_t+\varrho^kv^{k-1}\nabla v^{k}-\div(\nu(\theta^{k-1})\D( v^k))+\nabla \pi^k=0\\
&\varrho^k\theta^k_t+\varrho^kv^{k-1}\nabla\theta^{k}-\div(\kappa(\theta^{k-1}) \nabla \theta^k)=\nu(\theta^{k-1})\D^2 (v^{k}),\\
&\div v^{k}=0,\\
\end{split}
\end{equation}
\begin{equation}
\theta^k|_{t=0}= 0, \quad v^k|_{t=0}= 0, \quad \varrho^k|_{t=0}=\varrho_0 \qquad \mbox{in }\Omega.
\end{equation}
Assume that $(v^{k-1}, \theta^{k-1})$ are given and 
 \begin{equation}\label{k-1}
  \Xi_{\un}^p(v^{k-1}) + \Xi_{\uk}^{p/2}(\theta^{k-1}) \leq M.
 \end{equation}
 Since $p$ is sufficiently large  ($>7$) providing 
$v^{k-1},\nabla v^{k-1}\in L_p(0,T;L_\infty(\Omega))$,
we solve \eqref{app}$_1$ using the standard theory for transport equation to find $\varrho^k$ and  conclude that for all $t\ge0$
\begin{equation}
\|1-\varrho^k(t)\|_{L_\infty(\Omega)}=\|1-\varrho_0\|_{L_\infty(\Omega)}.
\end{equation}

Then we solve a linear problem \eqref{app}$_2$,  \eqref{app}$_4$ to find $v^k$
such that $  \Xi_{\un}^p(v^{k})  \leq C$,  see~\cite{DaMu2009}, where the issue of  regularity of perturbed Stokes problem is addressed. We insert $v^k$ in \eqref{app}$_3$
and by means of the same estimates as the ones provided  in Section~3  conclude 
that $\Xi_{\uk}^{p/2}(\theta^{k}) \leq C$.
By induction we claim for each $k\in{\mathbb N}$
the existence of solution $(\varrho^k, v^k, \theta^k)$ to the above problem. 

To show that the  mapping  $(\cdot)^{(n)} \to (\cdot)^{(n+1)}$ given by (\ref{app}) is  a contraction, we are required to have some information about the regularity of the density. However we have
\begin{equation}
 \frac{d}{dt}\|\nabla \rho^k\|_{L_p(\Omega)}\leq C\|\nabla v^{k-1}\|_{L_\infty(\Omega)} \|\nabla \rho^k\|_{L_p(\Omega)},
\end{equation}
so
\begin{equation}
 \sup_{t\leq T} \|\nabla \rho^k(t)\|_{L_p(\Omega)} \leq \|\nabla \rho_0\|_{L_p(\Omega)} \exp\{ \int_0^T C\|\nabla v^{k-1}(t)\|_{L_\infty} dt\} \leq \|\nabla \rho_0\|_{L_p(\Omega)} C(T).
\end{equation}
The bound does not depend on $k$, but it depends on $T$. We could remove this dependence analyzing in more detail the long time behavior of the solution. Nevertheless for the issue of existence
such dependence is not problematic.


\smallskip

For $\varphi=\varrho, v$ or $\pi$ we use the notation
\begin{equation}
\delta \varphi^{m}:=\varphi^m-\varphi^{m-1},\ m\in{\mathbb N}.
\end{equation}
Then $\delta\varrho^k$ satisfies
\begin{equation}\label{mass-delta}
\delta\varrho^k_t+v^{k-1}\nabla \delta\varrho^k=-\delta v^{k-1}\nabla \varrho^{k-1}
\end{equation}
and multiplying the above equation by $\delta \varrho^k$, integrating   over $(0,T)\times\Omega$, using that $\div v^{k-1}=0$   and estimating the right-hand side gives
\begin{equation}
\begin{split}
\frac{1}{2}\|\delta\varrho^k(t)\|^2_{L_2(\Omega)}&\le\int_0^t\|\delta v^{k-1}\|_{L_6(\Omega)}\|\nabla\varrho^{k-1}\|_{L_3(\Omega)}\|\delta\varrho^k\|_{{L_2(\Omega)}}\ d\tau\\
&\le\frac{1}{4}\|\delta\varrho^k\|_{L_\infty(0,T;L_2)}^2+C(T)\|\nabla \varrho^k_0\|_{L_p(\Omega)}^2
\int_0^t\|\D(\delta v^{k-1})\|_{L_2(\Omega)}^2\ d\tau, 
\end{split}
\end{equation}
hence 
\begin{equation}
\sup_{0\leq \tau\leq t}\|\delta\varrho^k(\tau)\|^2_{L_2(\Omega)}\le C(T)\int_0^t\|\D(\delta v^{k-1})\|_{L_2(\Omega)}^2\ d\tau.
\end{equation}
The momentum balance gives
\begin{equation}\label{mom-delta}
\varrho^k\delta v^k_t -\div(\nu(\theta^{k-1})\D(\delta v^k))+\nabla \delta\pi^k
={\mathcal R}^k,
\end{equation}
where the remainder term has the following form
\begin{equation}
\begin{split}
{\mathcal R}^k=&-\delta\varrho^kv^{k-1}_t-
\varrho^k v^{k-1}\nabla\delta v^k+\nabla v^{k-1}\delta\varrho^k v^{k-1}\\
&-\nabla v^{k-1}\varrho^{k-1}\delta v^{k-1}
+\div((\nu(\theta^{k-2})-\nu(\theta^{k-1}))\D(v^{k-1})).
\end{split}
\end{equation}
Observe that using \eqref{mass-delta}  one obtains
\begin{equation}\label{calc}
\begin{split}
\int_0^t\int_\Omega\varrho^k\delta  v^k_t\delta v^k \ dx\ d\tau&=\frac{1}{2}\int_0^t\int_\Omega\varrho^k[(\delta  v^k)^2]_t  \ dx\ d\tau=\frac{1}{2}\int_\Omega \varrho^k|\delta v^k|^2\ dx-\frac{1}{2}\int_0^t\int_\Omega\varrho^k_t|\delta v^k|^2\ dx\ d\tau\\
&=\frac{1}{2}\int_\Omega \varrho^k|\delta v^k|^2\ dx+\frac{1}{2}\int_0^t\int_\Omega\div (v^{k-1}\varrho^k)|\delta v^k|^2\ dx\ d\tau\\
&=\frac{1}{2}\int_\Omega \varrho^k|\delta v^k|^2\ dx-\frac{1}{2}\int_0^t\int_\Omega v^{k-1} \varrho^k
\nabla (|\delta v^k|^2)\ dx\ d\tau.
\end{split}
\end{equation}
Multiplying \eqref{mom-delta} by $\delta v^k$, integrating over $(0,T)\times\Omega$  and using \eqref{calc} gives
\begin{equation}\label{calc1}
\frac{1}{2}\int_\Omega \varrho^k|\delta v^k|^2\ dx+\int_0^t\int_{\Omega} \nu(\theta^{k-1})\D^2(\delta v^k) \ dx\ dt=
\int_0^t\int_{\Omega}  \left(v^{k-1}\varrho^k\nabla\delta v^k\delta v^k +{\mathcal R}^k \delta v^k\right)\ dx\ dt.
\end{equation}

Note that the first term on the right-hand side of \eqref{calc1} cancels with the the second term of the remainder ${\mathcal R}$. 
To estimate the right-hand side we proceed as follows
\begin{equation}
\begin{split}
\int_0^t\int_\Omega |\delta\varrho^k v^{k-1}_t\delta v^k|\ dx \ d\tau&\le
C\int_0^t\|\delta\varrho^k\|_{L_2(\Omega)}\| v^{k-1}_t\|_{L_3(\Omega)}\|\delta v^k\|_{L_6(\Omega)}\ d\tau\\
&\le  C(T)\left(t\int_0^t\|\D(\delta v^{k-1})\|_{L_2(\Omega)}^2\ d\tau\right)^{1/2} \|v_t^{k-1}\|_{L_2(0,T;L_3)}
\|\D(\delta v^k)\|_{L_2((0,T)\times\Omega)}\\
&\le C(T)t \int_0^t\|\D(\delta v^{k-1})\|_{L_2(\Omega)}^2\ d\tau+ \un/4\|\D(\delta v^k)\|^2_{L_2((0,T)\times\Omega)},
\end{split}
\end{equation}

\begin{equation}
\begin{split}
\int_0^t\int_\Omega |\nabla v^{k-1}\delta \varrho^k v^{k-1}\delta v^k|\ dx \ d\tau&\le
C\int_0^t \|\nabla v^{k-1}\|_{L_\infty(\Omega)}\|\delta \varrho^k \|_{L_2(\Omega)}\|v^{k-1}\|_{L_3(\Omega)}\|\delta v^k\|_{L_6(\Omega)} \ d\tau\\
&\le C(T)t \int_0^t\|\D(\delta v^{k-1})\|_{L_2(\Omega)}^2\ d\tau+ \un/4\|\D(\delta v^k)\|^2_{L_2((0,T)\times\Omega)},
\end{split}
\end{equation}

\begin{equation}
\begin{split}
\int_0^t\int_\Omega& | \nabla v^{k-1}\varrho^{k-1}\delta v^{k-1} \delta v^k|\ dx \ d\tau\le
C\int_0^t \| \nabla v^{k-1}\|_{L_\infty(\Omega)}\|\varrho^{k-1}\|_{L_\infty(\Omega)}\|\delta v^{k-1}\|_{L_2(\Omega)} \|\delta v^k\|_{L_2(\Omega)} \ d\tau\\
&\le {\un}/{4}\|\D(\delta v^k)\|_{L_2((0,T)\times\Omega)}^2+ C
\| \nabla v^{k-1}\|^2_{L_p(0,T;L_\infty)}\|\varrho^{k-1}\|^2_{L_\infty((0,T)\times\Omega)}
\|\D(\delta v^{k-1})\|^2_{L_2((0,T)\times\Omega)}
\end{split}
\end{equation}

\begin{equation}
\begin{split}
\int_0^t\int_\Omega &|(\nu(\theta^{k-2})-\nu(\theta^{k-1}))\D(v^{k-1})\D(\delta v^k)|\ dx \ d\tau\le
C\int_0^t\int_\Omega|\delta\theta^{k-1}|\cdot|\D(v^{k-1})\D(\delta v^k)|\ dx \ d\tau\\
&\le C\int_0^t\|\delta\theta^{k-1}\|_{L_6(\Omega)}\|\D(v^{k-1})\|_{L_3(\Omega)}\|\D(\delta v^k)\|_{L_2(\Omega)}d\tau\\&
\le C\|\delta\theta^{k-1}\|_{L_2(0,T;L_6)}\|\D(v^{k-1})\|_{L_\infty(0,T;L_3)}\|\D(\delta v^k)\|_{L_2((0,T)\times\Omega)}\\
&\le {\un}/{4}\|\D(\delta v^k)\|_{L_2((0,T)\times\Omega)}^2+ C\|\delta\theta^{k-1}\|^2_{L_2(0,T; L_6)}\|\D(v^{k-1})\|^2_{L_\infty(0,T;L_3)}.
\end{split}
\end{equation}
Finally we arrive at 
\begin{equation}\label{est-delta-v}
2\|\delta v^k(t)\|^2_{L_2(\Omega)}+\un\int_0^t\|\D(\delta v^k)\|^2_{L_2(\Omega)}\ d\tau\le
C(\|\D(\delta \theta^{k-1})\|^2_{L_2((0,T)\times\Omega)}+\|\D(\delta v^{k-1})\|_{L_2((0,T)\times\Omega)}^2).
\end{equation}
The heat equation for $\delta\theta^k$ reads as follows
\begin{equation}\label{temp-delta}
\varrho^k\delta\theta^k_t-\div(\kappa(\theta^{k-1})\nabla \delta\theta^k)={\mathcal J^k}
\end{equation}
and 
\begin{equation}
\begin{split}
{\mathcal J^k}:=& -\delta\varrho^k\theta^{k-1}_t
-
\varrho^kv^{k-1}\nabla (\delta \theta^k)
-\varrho^{k-1}\delta v^{k-1}\nabla \theta^{k-1}\\
&-
\delta \varrho^k v^{k-1}\nabla \theta^{k-1}
+\div((\kappa(\theta^{k-2})-\kappa(\theta^{k-1})))\nabla(\theta^{k-1})\\
&+\nu(\theta^{k-1})\D(\delta v^k)(\D(v^k)+\D(v^{k-1}))
+(\nu(\theta^{k-1})-\nu(\theta^{k-2}))\D^2(v^{k-1}).
\end{split}
\end{equation}
Multiplying \eqref{temp-delta} by $\delta \theta^k$, integrating over $(0,T)\times\Omega$ yields
\begin{equation}
\frac{1}{2}\int_\Omega \varrho^k|\delta \theta^k|^2\ dx+\int_0^t\int_{\Omega} \kappa(\theta^{k-1})|\nabla(\delta \theta^k) |^2\ dx\ d\tau=
\int_0^t\int_{\Omega}  \left(v^{k-1}\varrho^k\nabla\delta \theta^k\delta \theta^k +{\mathcal J}^k \delta \theta^k\right)\ dx\ d\tau.
\end{equation}
Again we estimate the right-hand side
\begin{equation}
\begin{split}
\int_0^t\int_\Omega|\delta\varrho^k\theta^{k-1}_t\delta \theta^k|\ dx\ d\tau&\le
C\int_0^t\|\delta\varrho^k\|_{L_2(\Omega)}\| \theta^{k-1}_t\|_{L_3(\Omega)}\|\delta \theta^k\|_{L_6(\Omega)}\ d\tau\\
&\le  C(T)\left(t\int_0^t\|\D(\delta v^{k-1})\|_{L_2(\Omega)}^2\ d\tau\right)^{1/2} \|\theta_t^{k-1}\|_{L_2(0,T; L_3)}
\|\D(\delta \theta^k)\|_{L_2((0,T)\times\Omega)}\\
&\le C(T)t \int_0^t\|\D(\delta v^{k-1})\|_{L_2(\Omega)}^2\ d\tau+ \uk/4\|\D(\delta \theta^k)\|^2_{L_2( (0,T)\times\Omega)},
\end{split}
\end{equation}

\begin{equation}
\begin{split}
\int_0^t\int_\Omega|&
\varrho^{k-1}\delta v^{k-1}\nabla \theta^{k-1}\delta \theta^k|
\ dx\ d\tau\le
C\int_0^t\|\varrho^{k-1}\|_{L_\infty(\Omega)}\|\delta v^{k-1}\|_{L_2(\Omega)}\|\nabla \theta^{k-1}\|_{L_\infty(\Omega)}\|\delta \theta^k\|_{L_2(\Omega)}\\
&\le\uk/4\|\nabla (\delta\theta^k)\|^2_{L_2((0,T)\times\Omega)}+C\|\varrho^{k-1}\|^2_{L_\infty((0,T)\times\Omega)}\|\nabla \theta^{k-1}\|^2_{L_{p/2}(0,T;L_\infty)}\|\delta v^{k-1}\|^2_{L_2((0,T)\times\Omega)},
\end{split}
\end{equation}



\begin{equation}
\begin{split}
\int_0^t\int_\Omega|
\delta \varrho^k v^{k-1}\nabla \theta^{k-1}\delta \theta^k|
\ dx\ d\tau&\le
C\int_0^t \|\delta \varrho^k\|_{L_2(\Omega)} \|v^{k-1}\|_{L_\infty(\Omega)}\|\nabla \theta^{k-1}\|_{L_3(\Omega)}\|\delta \theta^k\|_{L_6(\Omega)}\ d\tau\\
&\le  C(T)t \int_0^t\|\D(\delta v^{k-1})\|_{L_2(\Omega)}^2\ d\tau+ \uk/4\|\nabla(\delta \theta^k)\|^2_{L_2((0,T)\times\Omega)},
\end{split}
\end{equation}

%
%
%

\begin{equation}
\begin{split}
\int_0^t\int_\Omega|&
((\kappa(\theta^{k-2})-\kappa(\theta^{k-1})))\nabla(\theta^{k-1})\nabla (\delta \theta^k)|
\ dx\ d\tau\le C\int_0^t\int_\Omega|\delta \theta^{k-1}\nabla (\theta^{k-1})\nabla (\delta \theta^k)|
\ dx\ d\tau\\
&\le C\int_0^t \frac{\|\kappa'(\theta)\|_{L_\infty((0,T)\times\Omega)}}{\uk^{1/2}} |\delta \theta^{k-1}| \uk^{1/2}|\nabla \theta^k||\nabla \delta \theta^k|d\tau\\
& \le C\frac{\|\kappa'(\theta)\|_{L_\infty((0,T)\times\Omega)}}{\uk^{1/2}} \|\delta \theta^{k-1}\|_{L_2(0,T;L_6)}\kappa^{1/2}\|\nabla \delta \theta^k\|_{L_2(0,T;L_2)} 
\|\nabla \theta^k\|_{L_\infty(0,T;L_3)} \\
&\le \uk/4\|\nabla(\delta\theta^k)\|^2_{L_2((0,T)\times\Omega)}+CM
\left(\frac{\|\kappa'(\theta)\|_{L_\infty((0,T)\times\Omega)}}{\uk}\right)^2 \uk \|\nabla (\delta \theta^{k-1})\|^2_{L_2((0,T)\times\Omega)},
\end{split}
\end{equation}
where $M$ is the constant given in \eqref{k-1}. We continue with the remaining terms

\begin{equation}
\begin{split}
\int_0^t\int_\Omega|&
\nu(\theta^{k-1})\D(\delta v^k)(\D(v^k)+\D(v^{k-1}))\delta\theta^k|
\ dx\ d\tau\\&\le
C\int_0^t\|\nu(\theta^{k-1})\|_{L_\infty(\Omega)}\|\D(\delta v^k)\|_{L_2(\Omega)}\|\D(v^k)+\D(v^{k-1})\|_{L_3(\Omega)}\|\delta\theta^k\|_{L_6(\Omega)}\ d\tau\\
&\le
C\|\nu(\theta^{k-1})\|_{L_\infty((0,T)\times\Omega)}
\|\D(v^k)+\D(v^{k-1})\|_{L_\infty(0,T;L_3)}^2\|\D(\delta v^k)\|^2_{L_2((0,T)\times\Omega)}\\
&+\uk/4\|\nabla(\delta\theta^k)\|^2_{L_2((0,T)\times\Omega)},
\end{split}
\end{equation}

\begin{equation}
\begin{split}
\int_0^t\int_\Omega|
(\nu(\theta^{k-1})-\nu(\theta^{k-2}))&\D^2(v^{k-1})\delta\theta^k|\ dx\ d\tau\le C\int_0^t\int_\Omega |\delta\theta^{k-1}\D^2(v^{k-1})\delta\theta^k|
\ dx\ d\tau\\
&\le \int_0^T\|\delta\theta^{k-1}\|_{L_3(\Omega)}\|\D^2(v^{k-1})\|_{L_2(\Omega)}\|\delta \theta^k\|_{L_6(\Omega)}\ d\tau\\
&\le  \|\D^2(v^{k-1})\|^2_{L_2((0,T)\times\Omega)}\|\delta\theta^{k-1}\|^2_{L_2(0,T;W^1_{p/2})}+\uk/4\|\delta \theta^k\|^2_{L_2(0,T;W^1_{p/2})},
\end{split}
\end{equation}

\begin{equation}\begin{split}
\|\delta \theta^k(t)\|^2_{L_2(\Omega)}+\uk\int_0^t\|\nabla(\delta \theta^k)\|^2_{L_2(\Omega)}\ d\tau&\le
 C M
\left(\frac{\|\kappa'(\theta)\|_{L_\infty((0,T)\times\Omega)}}{\uk}\right)^2 \uk
\|\nabla(\delta \theta^{k-1})\|^2_{L_2((0,T)\times\Omega)}\\&+ C\|\D(\delta v^{k})\|_{L_2((0,T)\times\Omega)}^2.
\end{split}\end{equation}

Now we fix $t$ so small that $C(T)t \leq 1$. 
Multiplying \eqref{est-delta-v} by a sufficiently large constant to absorb the term $\|\D(\delta v^k)\|^2_{L_2((0,T)\times\Omega)}$ we obtain

\begin{equation}
\begin{split}
2\|\delta v^k(t)\|^2_{L_2(\Omega)}+2\|\delta \theta^k(t)\|^2_{L_2(\Omega)}+\un \int_0^t\|\D(\delta v^k)\|^2_{L_2(\Omega)}\ d\tau
+\uk \int_0^t\|\nabla(\delta \theta^k)\|^2_{L_2(\Omega)}\ d\tau\\
\le
\frac 12  \big( \int_0^t \un \|\D(\delta v^{k-1})\|_{L_2(\Omega)}^2 d\tau +
\uk \int_0^t \|\nabla (\delta \theta^{k-1})\|^2_{L_2(\Omega)}\big).
\end{split}
\end{equation}

This way we proved the convergence of  the sequence 
$(\varrho^k,  v^k, \theta^k)$ in  space $L_\infty(0,T;L_2(\Omega))$, hence we get the strong convergence in all the spaces intermediate with the
ones defined by the a priori estimates. But also we have weak and weak-$\ast$ compactness, thus we find that
the limit must fulfill the original system (\ref{main}). The existence is established on the time interval $[0,t]$. Since all constants in the above considerations 
depended only on $T$, so proceeding step by step on time intervals $[kt,(k+1)t]$ we obtain the existence for the whole interval $[0,T]$. We shall just recall that $T$ is an arbitrary 
large number fixed at the beginning of our  proof of existence. Thus, we conclude that the constructed solutions exist globally in time, since they obey the a priori estimate.

Quasilinearity of the system removes all questions
concerning the convergence of nonlinear terms. Of course, the contraction in the large space yields also uniqueness to the system. The main theorem is proved.

\end{proof}

We complete our analysis of  system (\ref{main}) by stating  the last result concerning the long time behavior of solutions.

\begin{theorem}\label{T:asympt}
Let the assumptions of Theorem~\ref{Th:main} be satisfied. Then $(v,\theta)$ solving  \eqref{main} decay at infinity, i.e. there exist constants $\alpha, \beta, \gamma>0$  such that
\begin{equation}
\|v(t)\|_{L_2(\Omega)}\sim e^{-\un t}, \quad\|v(t)\|_{W^{2-2/p-\epsilon}_p(\Omega)}\sim e^{-\alpha\un t}, 
\end{equation}
and
\begin{equation}
\|\theta(t)-\ut- \tilde \theta_\infty \|_{L_2(\Omega)}
\sim e^{-\beta t}, \quad \|\theta(t)-\ut- \tilde \theta_\infty \|_{W^{2-4/p}_{p/2}(\Omega)}
\sim e^{-\gamma t}
\end{equation}
for a constant $\tilde \theta_\infty>0$.  Moreover, $\varrho$ solving  \eqref{main} satisfies
\begin{equation}\label{rho-asa}
\sup_t\|\nabla \varrho(t)\|_{L_p(\Omega)} < \infty.
\end{equation}
\end{theorem}

\begin{proof}
Estimate \eqref{k0} allows to conclude that
\begin{equation}
\|v(t)\|_{L_2(\Omega)}\sim e^{-\un t} \quad \mbox{and}\quad \|v(t)\|_{W^{2-2/p}_p(\Omega)}\le c.
\end{equation}
Hence by interpolation we conclude that there exists a constant $\alpha>0$
\begin{equation}\label{asa}
\|v(t)\|_{W^{2-2/p-\epsilon}_p(\Omega)}\sim e^{-\alpha\un t}.
\end{equation}
In the next step we will use this information to conclude the asymptotic behavior of the temperature. 
Looking at $(\ref{main})_4$, having (\ref{asa}), we have
\begin{equation}
 \frac{d}{dt}\int_\Omega \varrho(\theta -\ut) dx = \int_\Omega \nu(\theta)\D^2(v) dx,
\end{equation}
thus
\begin{equation}
 \int_\Omega \varrho(\theta(t)-\ut)dx=\int_\Omega \varrho_0 (\theta_0-\ut)dx + \int_0^t \int_\Omega \nu (\theta)\D^2(v) dxd\tau=: \int_\Omega \varrho \tilde \theta(t) dx
\end{equation}
where $\tilde \theta(\cdot)$ is a function of time given by
\begin{equation}
 \tilde \theta(t)=\frac{1}{{\rm mass}} \big( \int_\Omega \varrho_0 (\theta_0-\ut)dx + \int_0^t \int_\Omega \nu (\theta)\D^2(v) dxd\tau\big).
\end{equation}
Then the energy equation is restated as follows
\begin{multline}
 \varrho (\theta -\ut - \tilde\theta)_t + \varrho v \cdot \nabla (\theta -\ut - \tilde\theta) -
\div \kappa(\theta)\nabla (\theta -\ut - \tilde\theta)=\\
\nu(\theta)\D^2(v) - \{\nu(\theta)\D^2(v)\}.\qquad
\end{multline}
From \eqref{asa} we estimate the right-hand side of the heat equation, $\nu(\theta)\D^2(v)\sim e^{-\alpha\un t}$.  Then
\begin{equation}
\|\theta(t)-\ut -\tilde \theta\|_{L_2(\Omega)}
\sim e^{-\beta t}, 
\quad
\|\theta(t)-\ut\|_{W^{2-4/p}_{p/2}(\Omega)}
\sim e^{-\gamma t}.
\end{equation}
Note we  used the Poincar\'e inequality, since our construction yields
\begin{equation}
 \int_\Omega \varrho (\theta -\ut - \tilde\theta)dx =0 \mbox{ for all t}.
\end{equation}
The constant $\tilde \theta_\infty = \tilde \theta(\infty)$.

To observe that  \eqref{rho-asa} holds, we differentiate equation \eqref{main}$_{(1)}$ to obtain
\begin{equation}
\frac{d}{dt}\|\nabla \varrho\|_{L_p(\Omega)}\le \|\nabla v\|_{L_\infty(\Omega)}\|\nabla \varrho\|_{L_p(\Omega)}.
\end{equation}
With help of Gronwall lemma we conclude that 
\begin{equation}
\sup_t\|\nabla \varrho\|_{L_p(\Omega)}\le \|\nabla \varrho_0\|_{L_p(\Omega)} \exp\left(\int_0^\infty\|\nabla v\|_{L_\infty(\Omega)}\ dt\right).
\end{equation}
Hence  by (\ref{asa}) the proof is complete. 
\end{proof}

{\bf Acknowledgments.} The work of both authors has been partly supported by Polish NCN grant No  2014/13/B/ST1/03094.

\bibliographystyle{abbrv}
\bibliography{budyn}
\end{document}